\documentclass[a4paper, 12pt]{article}
\usepackage[dvips]{graphicx}
\usepackage{amssymb}
\usepackage{amsmath}
\usepackage{amsthm}
\usepackage{amscd}
\usepackage[active]{srcltx}
\theoremstyle{plain}
 \newtheorem{theorem}{Theorem}[section]
 \newtheorem*{theorem*}{Theorem}
 \newtheorem*{lemma*}{Lemma}
 \newtheorem{proposition}[theorem]{Proposition}
 
 \newtheorem{fact*}{Fact}
 \newtheorem{lemma}[theorem]{Lemma}
 
\theoremstyle{definition}

 \newtheorem*{remark*}{Remark}
 
 \newtheorem{example}[theorem]{Example}
\numberwithin{equation}{section}
\numberwithin{figure}{section}

\usepackage[usenames]{color}
\newcommand{\red}[1]{\textcolor{black}{#1}}

\newcommand{\vect}[1]{\boldsymbol{#1}}
\newcommand{\R}{\boldsymbol{R}}

\newcommand{\rank}{\operatorname{rank}}
\newcommand{\Hess}{\operatorname{Hess}}

\renewcommand{\phi}{\varphi}

\newcommand{\inner}[2]{\left({#1},{#2}\right)}
\newcommand{\innera}[2]{\left\langle{#1},{#2}\right\rangle}

\renewcommand{\span}{\operatorname{span}}

\newcommand{\zv}{\vect{0}}
\renewcommand{\dfrac}{\displaystyle\frac}
\newcommand{\pmt}[1]{{\begin{pmatrix} #1  \end{pmatrix}}}

\newcommand{\trans}[1]{{\vphantom{#1}}^t{#1}}
\makeatletter
\renewcommand{\section}{%
  \@startsection{section}% #1 name
   {1}% #2 level
   {\z@}% #3 indent length
   {-3.5ex \@plus -1ex \@minus -.2ex}% #4 up space
   {2.3ex \@plus.2ex}% #5 under space
   {\normalfont\normalsize\bfseries}% #6
}%
\makeatother%
\makeatletter
\renewcommand{\subsection}{%
  \@startsection{subsection}% #1 name
   {1}% #2 level
   {\z@}% #3 indent length
   {-3.5ex \@plus -1ex \@minus -.2ex}% #4 up space
   {2.3ex \@plus.2ex}% #5 under space
   {\normalfont\normalsize\bfseries}% #6
}%
\makeatother
\begin{document}
\begin{center}
{\large
{\bf
Characterizing singularities of a surface in Lie sphere geometry}
}\\
Mason Pember,
Wayne Rossman,
Kentaro Saji
and
Keisuke Teramoto\\
\today
\end{center}
\begin{abstract}
The conditions for a cuspidal edge, swallowtail
and other fundamental singularities
are given in the context of Lie sphere geometry. 
We then use these conditions to
study the Lie sphere transformations of a surface. 

%The conditions for 
%a Legendre immersion to project to a cuspidal edge, swallowtail,
%cuspidal beaks/lips, cuspidal butterfly
%or $D_4^\pm$-singularity are given in a Lie geometric way.  
%The conditions are completely the same as
%the Euclidean case.
\end{abstract}
%\tableofcontents

\section{Introduction}

Lie sphere geometry is the study of spheres in spaceforms and their 
tangential 
contact.  It was first developed by 
Sophus Lie \cite{lie} using his hexaspherical 
coordinate model.  This model has been utilized by 
Blaschke \cite{blaschke}
to study applicable surfaces in Lie sphere geometry, 
and the recent interest in 
integrable systems has revived this approach 
(see, for example, \cite{clarke, 
ferapontov1, ferapontov2, mussonic, mason}).  Spacelike linear 
Weingarten surfaces with singularities 
in spaceforms can be characterized in this setting 
(\cite{BHR1, BHR2}), and that 
is one motivation for 
the present work, in which we investigate 
relationships between Lie sphere geometry and 
the singularity theory of surfaces.  \red{In 
particular, we obtain the following: 
\begin{itemize}
\item In Theorem \ref{thm:singcond}, 
we characterize certain corank $1$ singularities 
arising from non-umbilic points of Legendre 
immersions in Lie sphere geometry, 
those singularities being 
cuspidal edges, swallowtails, cuspidal lips, cuspidal beaks, 
cuspidal butterflies, and more 
degenerate singularities of certain types.
\item In Theorem \ref{thm:singcondu}, 
we characterize certain corank $2$ 
singularities arising from umbilic points 
of Legendre immersions, those being 
$D_4^\pm$ singularities.
\item We classify these singularities by certain differential 
properties, and in Theorem \ref{thm:singcondl} 
and Theorem \ref{lastthm} we show how Lie 
sphere transformations preserve these classes.  
\end{itemize}}
For simplicity of exposition we shall only consider 
singularities of surfaces in Euclidean 3-space, however, 
by employing analogous arguments as those used in this 
paper one can show that our results hold in any 
3-dimensional Riemannian spaceform.

\section{\red{Preliminaries on Lie sphere geometry}}
\label{sec:LiePrelims}

In this section we will explain the 
notions and terminologies from Lie sphere geometry that will 
be used in this paper.

Let $\R^{4,2}$ be a 6-dimensional vector space
equipped with the inner product $\inner{~}{~}=(-++++-)$. 
Let
\[ L^5=\{ X \in \R^{4,2} \, | \, 
(X,X)=0 \}\] 
denote the lightcone of $\R^{4,2}$. 

Choosing nonzero $p\in\R^{4,2}$ and $q\in\R^{4,2}$ with 
$p \perp q$ and $p$ non-null,
let us define the following $3$-dimensional quadrics
\begin{equation}
\label{eqn:spforms}
\begin{array}{rcl}
M &=& \{ X \in L^{5} \, | \, (X,p)=0, (X,q)=-1 \}\subset\R^{4,2},\\
N &=& \{ X \in L^{5} \, | \, (X,q)=0, (X,p)=-1 \}\subset\R^{4,2}.
\end{array}
\end{equation}
We call $p$ the {\it point sphere complex\/} 
and $q$ the {\it spaceform vector\/} of $M$. If $p$ is timelike 
(respectively, spacelike), then $M$ is isometric to a 3-dimensional 
Riemannian (respectively, Lorentzian) spaceform with constant 
sectional curvature $\kappa = -(q,q)$. For example, if $\inner{p}{p}=-1$ 
and  $q$ is null, then $M$ is isometric to
a Euclidean space $\R^3$. Cecil~\cite[Section 2.3]{cecil} made this 
identification explicit via the isometry 
\begin{equation}\label{eq:correuc}
\xi\in \R^{3} \mapsto \trans{\left(\frac{1+\innera{\xi}{\xi}}{2}, 
\frac{1-\innera{\xi}{\xi}}{2},\xi,0\right)}\in M,
\end{equation}
where $\innera{~}{~}$ denotes the inner product of $\R^{3}$. 
In this case $M$ is determined by $p = \trans(0,0,0,0,0,1)$ 
and $q=\trans(1,-1,0,0,0,0)$. 

Points in the projective light cone $\mathbb{P}(L^{5})$ 
of $\R^{4,2}$ 
correspond to spheres in spaceforms in the following way. 
Let $s\in\mathbb{P}(L^{5})$, then we let 
\[ S:= M\cap s^{\perp}.\]
If $s\not\perp p$, i.e., $(\sigma,p)\neq 0$ for any 
$\sigma\in L^{5}$ such that  $s=\span\{\sigma\}$, then 
the set of points determined by $s$ is a metric 
sphere or a plane in $M$. Otherwise $S$ is a point in $M$. 
For example, suppose we are using the identification of $M$ 
with $\R^{3}$ in  \eqref{eq:correuc}, where $p = \trans(0,0,0,0,0,1)$ 
and $q=\trans(1,-1,0,0,0,0)$. Write $\sigma\in L^{5}$ such that 
$s=\span\{\sigma\}$ as $\sigma =\trans(a,b,\zeta,c)$, 
where $a,b,c\in \R$ and $\zeta\in \R^{3}$. 
Then $\gamma=\trans{\left(\frac{1+\innera{\xi}{\xi}}{2}, 
\frac{1-\innera{\xi}{\xi}}{2},\xi,0\right)} \in S$ if and only if 
\begin{equation}\label{eq:spherereason}
0= (\gamma,\sigma)= -a\dfrac{1+\innera{\xi}{\xi}}{2}
+b\dfrac{1-\innera{\xi}{\xi}}{2}
+\innera{\zeta}{\xi} = \dfrac{1}{2}(b-a) - 
\dfrac{\innera{\xi}{\xi}}{2}(b+a) + \innera{\zeta}{\xi}.
\end{equation}
Furthermore, since $s\in \mathbb{P}(L^{5})$, 
\begin{equation}\label{eq:spherereason2}
0= -a^2+b^2+\innera{\zeta}{\zeta}-c^2= (a+b)(b-a) +\innera{\zeta}{\zeta}-c^2. 
\end{equation}
Now if $s\not\perp q$, i.e., $(\sigma, q)= a+b\neq 0$, then we may scale 
$\sigma$ so that $a+b=1$. Then \eqref{eq:spherereason} and 
\eqref{eq:spherereason2} imply that $\gamma \in S$ if and only if
\[ \innera{\xi - \zeta}{\xi - \zeta} = c^{2}.\]
Hence, $S$ is a sphere of radius $|c|$ with center $\zeta$. 
If $s\perp q$, i.e, $(\sigma,q)=a+b=0$, then we may scale 
$\sigma$ so that $c^{2}=1$. Then \eqref{eq:spherereason2} 
implies that $\zeta\in \R^{3}$ has unit length and one has 
that \eqref{eq:spherereason} is equivalent to 
\[  \innera{\zeta}{\xi} =  \dfrac{1}{2}(b-a).\]
Hence, $S$ defines a plane with unit normal $\zeta$. 
For a more in depth explanation, see \cite[Section 2.3]{cecil}. 

Now suppose that $s_{1},s_{2}\in \mathbb{P}(L^{5})$ and let 
$S_{i}:=M\cap s_{i}^{\perp}$ for $i\in\{1,2\}$ denote the 
corresponding spheres in $M$. Then $S_{1}$ and $S_{2}$ are 
in oriented tangential contact if and only if $s_{1}\perp s_{2}$. 
The maximal dimension of null subspaces of $\R^{4,2}$ is 2. 
Let $\mathcal{Z}$ denotes the Grassmannian of these null 2-dimensional 
subspaces. Alternatively, we can think of $\mathcal{Z}$ as 
the space of lines in $\mathbb{P}(L^{5})$. Then given 
$\mathcal{P}\in \mathcal{Z}$, any two elements $s_{1},s_{2}\in \mathcal{P}$ 
satisfy $s_{1}\perp s_{2}$ and thus $\mathcal{P}$ corresponds to 
a pencil of spheres in oriented tangential contact in any spaceform $M$. 
Thus $\mathcal{P}$ is referred to as a \textit{contact element}.
We will see that $\mathcal{Z}$ is a contact manifold in the
next section.

\subsection{Legendre immersions}
In this section we shall recall from (\cite[Section 4.1]{cecil}) 
how the notion of Legendre immersion from the perspective of Lie sphere
geometry coincides with its more well known analogue using the unit 
tangent bundle $T_{1}S^{3}$ of $S^{3}$. To achieve this, 
let $p =\trans(0,0,0,0,0,1)$ and $q=\trans(1,0,0,0,0,0)$. 
Then we have that both $M$ and $N$ are isometric to $S^{3}$ by the maps
\[ \zeta\in S^{3} \mapsto \overline{\zeta}=\trans(1,\zeta,0) \in M, 
\quad \xi \in S^{3}\mapsto \overline{\xi}=\trans(0,\xi,1)\in N.\]
Then a pair $(\zeta, \xi)\in S^{3}\times S^{3}$ takes values in $T_1S^3=
\{(\zeta,\xi)\in S^3\times S^3\,|\,\innera{\zeta}{\xi}=0\}$, 
where $\innera{~}{~}$ stands for
the Euclidean inner product of $\R^4$, if and only if 
$\span\{\overline{\zeta}, \overline{\xi}\}$ takes values in $\mathcal{Z}$. 
Thus we obtain a bijective map $C:T_1S^3\to\mathcal{Z}$,
and we derive
a differentiable structure on $\mathcal{Z}$ from $C$.
Since $T_1S^3$ has a standard contact structure,
we also derive a contact structure
on $\mathcal{Z}$ by $C$.

Let $\mathcal{F}$ be a map from a $2$-dimensional 
manifold $\Sigma$ to $\mathcal{Z}$.
Since $\mathcal{F}(x)$ is a $2$-dimensional null 
subspace of $\R^{4,2}$ for each $x\in \Sigma$,
we can think of $\mathcal{F}$ as a rank $2$ null subbundle of the 
trivial bundle $\Sigma\times \R^{4,2}$ over $\Sigma$.
Any independent pair of sections $\sigma_1$, 
$\sigma_2\in \Gamma\mathcal{F}$, 
such that  $\mathcal{F}=\span\{\sigma_{1},\sigma_{2}\}$, will 
be called a {\em 
generator\/} of $\mathcal{F}$. Let $F:= \mathcal{F}\cap M$ and 
$T:= \mathcal{F}\cap N$. Then according to our identifications 
we have $F=\trans(1,f,0)$ and $T=\trans(0,t,1)$, for some maps 
$f,t:\Sigma\to S^{3}$. 
Now $C^{-1}\circ \mathcal{F}=(f,t):\Sigma\to T_{1}S^{3}$ is said 
to be a Legendre immersion if the map is an immersion and $(f,t)$ 
is isotropic for the contact structure on $T_{1}S^3$, i.e., 
$\innera{df}{t}=0$. It is then clear that the isotropy condition 
is equivalent to $(dF,T)=0$. One can then easily deduce that this 
holds if and only if for any generators 
$\sigma_{1},\sigma_{2}\in\Gamma \mathcal{F}$  
\[ (d\sigma_{1},\sigma_{2})=0.\]
It remains to characterize the condition that 
$\mathcal{C}^{-1}\circ \mathcal{F}$ is an immersion:
\begin{lemma}
\label{lem:immcond}
The map\/
$\mathcal{C}^{-1}\circ\mathcal{F}:\Sigma\to\mathcal{Z}$
is an immersion if and only if \/
for all \/ $x\in\Sigma$
and\/ $X\in T_x\Sigma$, 
\begin{equation}\label{eq:immcond}
d_X\sigma\in \mathcal{F}(x)
\text{ for all sections\/ }
\sigma\in \Gamma
\mathcal{F}
\text{ implies\/ }
X=0,
\end{equation}
where\/
$d:\Gamma T\Sigma\times\Gamma(\Sigma\times\R^{4,2})
\to\Gamma(\Sigma\times\R^{4,2})$
is the flat connection of\/ $\Sigma\times\R^{4,2}$.
\end{lemma}

\begin{proof}
Since $\mathcal{F}=\span\{F,T\}$, we may write any 
$\sigma\in \Gamma \mathcal{F}$ as 
$$
\sigma
=
\alpha
\pmt{1\\f\\0}
+
\beta
\pmt{0\\ t\\1}
$$ for some choice of coefficient functions $\alpha,\beta$. 
Now for any $x\in \Sigma$ and $X\in T_{x}\Sigma$, 
\[ d_{X}\sigma= d_{X}\alpha \pmt{1\\f(x)\\0} + \alpha(x) 
\pmt{0\\d_{X}f\\0} + d_{X}\beta \pmt{0\\ t(x) \\1} + \beta(x)
\pmt{0\\d_{X}t\\0}.\]
Therefore, $d_{X}\sigma\in \mathcal{F}(x)$ if and only if  
$\alpha(x)d_{X}f +\beta(x)d_{X}t =0$. This holds for arbitrary 
$\sigma$, i.e., for arbitrary $\alpha$ and $\beta$, if and only if 
$d_{X}f =d_{X}t = 0$.

On the other hand the map $C^{-1}\circ \mathcal{F}=(f,t)$ is
an immersion if and only if
$$
\ker df\cap \ker dt=\{0\}.
$$
Namely, $d_X f=d_X t=0$ implies $X=0$.
\end{proof}

Motivated by this result, one says that 
$\mathcal{F}:\Sigma\to \mathcal{Z}$ is a 
\textit{Legendre immersion} 
(in the context of Lie sphere geometry) if the following 
two conditions hold
\begin{enumerate}
\item for any sections $\sigma_{1},\sigma_{2}\in \Gamma \mathcal{F}$, 
$(d\sigma_{1},\sigma_{2})=0$,
\item for all $x\in \Sigma$ and $X\in T_{x}\Sigma$, if 
$d_{X}\sigma\in \mathcal{F}(x)$ for all sections 
$\sigma\in \Gamma \mathcal{F}$, then $X=0$. 
\end{enumerate}

We have already seen in this subsection how one can identify 
$\mathcal{F}$ with maps into the spaceform $S^{3}$. However, 
suppose now that we have a general point sphere complex $p$ and 
a spaceform vector $q$ defining quadrics $M$ and 
$N$ as in~(\ref{eqn:spforms}). Assume that
\begin{equation}\label{eq:assump0}
\det B\neq0,\quad B:=\pmt{
\inner{\sigma_{1}}{p}&\inner{\sigma_{2}}{p}\\
\inner{\sigma_{1}}{q}&\inner{\sigma_{2}}{q}}
\end{equation}
for a generator $\{\sigma_{1},\sigma_{2}\}$ of $\mathcal{F}$.
Then by \eqref{eq:assump0}, we can obtain 
a map $F:\Sigma\to M$ as an intersection of 
$\mathcal{F}$ with $M$,
and
a map $T:\Sigma\to N$ as an intersection of 
$\mathcal{F}$ with $N$ as follows:
\begin{equation}\label{eq:projeceuc}
\begin{array}{ll}
\begin{array}{rcl}
F(x)&=&a(x)\sigma_{1}(x)+b(x)\sigma_{2}(x),\\
T(y)&=&c(x)\sigma_{1}(x)+d(x)\sigma_{2}(x),
\end{array}
\quad
\pmt{a(x)&c(x)\\b(x)&d(x)}
=
B^{-1}\pmt{0&-1\\-1&0}
\end{array}
\end{equation}
By \cite[Theorem 4.2]{cecil},
$F$ can be interpreted as a projection of 
$\mathcal{F}$ to $M$ by a Legendre fibration.
Thus $F$ is a front in the sense of \cite{AGV}.
One then has that, at each point $x\in \Sigma$, $\mathcal{F}(x)$ 
corresponds to the pencil of oriented spheres tangent to the surface 
$F$ at $x$. Furthermore, if one consider the sphere defined by 
$\span\{T(x)\}$, this corresponds to the unique totally 
geodesic ``plane" in the sphere pencil $\mathcal{F}(x)$. 
Thus, for a Euclidean projection, this simply corresponds to 
the tangent plane of the surface $F$ at $x$. Therefore, 
we shall refer to $T$ as the \textit{tangent plane congruence} of $F$.

Using the specific projection of Cecil~\cite[Section 2.3]{cecil} to 
Euclidean 3-space $\R^3$, 
where
$p={}^t(0,0,0,0,0,1)$ and $q={}^t(1,-1,0,0,0,0)$,
$F$ and $T$ have the form
\begin{equation}\label{eq:lifts}
F=\trans{\left(
\dfrac{1}{2}(1+\innera{f}{f}),
\dfrac{1}{2}(1-\innera{f}{f}),
f,0\right)},\quad
T=\trans{\big(\innera{f}{t},-\innera{f}{t},t,1\big)},
\end{equation}
where $f=(f_1,f_2,f_3)$, $t=(t_1,t_2,t_3)$ 
and $\innera{a}{b}=a_1b_1+a_2b_2+a_3b_3$ 
$(a=(a_1,a_2,a_3), b=(b_1,b_2,b_3)$).
We then obtain a smooth map $f=(f_1,f_2,f_3):\Sigma\to\R^3$
and its unit normal vector $t=(t_1,t_2,t_3)$.
Although $\mathcal{F}$ is a Legendre immersion, $f$ may have singularities.
Such a projection is a front, and 
the study of singularities of fronts has a long history,
see \cite{AGV,acous} for example.
It is known that generic singularities of fronts
are cuspidal edges and swallowtails.
Useful criteria for them are given in \cite{krsuy},
and using these criteria, singularities of surfaces which
have special curvature properties are characterized by
their geometric properties.  (See
\cite{fsuy,ishima,izdual,mandala,horoflat,
harm,krsuy,orien,luy,mu,nao,coh,max} for example.) 
It is natural to ask whether
the properties of singularities of the front $f$
can be characterized using the 
Lie-geometric properties of $\mathcal{F}$.
Thus the purpose of this note is to study
singularities of $f$ in the context of Lie sphere geometry.
As an application, we shall study singularities of
Lie sphere transformations
of a regular surface in $\R^3$, as 
noted in the introduction.
Since Lie sphere transformations include the transformations to 
parallel surfaces,
this generalizes the results in \cite{FH}.

\subsection{Curvature sphere congruences}
Suppose that we have a 2-dimensional manifold $\Sigma$ and 
a smooth map $s:\Sigma\to \mathbb{P}(L^{5})$. 
Then at each point $x\in \Sigma$, $s(x)$ corresponds to a sphere 
in any given spaceform. Therefore, $s$ defines a 2-dimensional 
family of spheres in such spaceforms, i.e., a sphere congruence. 
Alternatively, we may think of such a map as a rank 1 null subbundle 
of the trivial bundle $\Sigma\times \R^{4,2}$. Now, given a 
Legendre immersion $\mathcal{F}:\Sigma\to \mathcal{Z}$, 
one may consider the rank 1 subbundles $s\le \mathcal{F}$. 
These correspond to sphere congruences enveloped by the surface 
defined by $\mathcal{F}$ in any spaceform.

At a point $x\in \Sigma$, we say that a $1$-dimensional subspace $s(x)$ of 
$\mathcal{F}(x)$ is a {\it curvature sphere}\/
if there exists a subspace $T_x\subset T_x\Sigma$ 
such that 
\begin{equation}\label{eq:curvsp}
\text{for any } \sigma\in \Gamma\mathcal{F} \text{ satisfying }
\sigma(x)\in s(x), 
\text{ and for any } 
X_x\in T_x,\ 
d_{X_x}\sigma\in\mathcal{F}(x).
\end{equation}
If $s(x)$ is a curvature sphere, then
take a maximal subspace $T_{s(x)}$ such that
\eqref{eq:curvsp} holds, and call $T_{s(x)}$
{\it curvature space\/} of $s(x)$.
Cecil~\cite[Corollary 4.9]{cecil} showed that at each point 
$x\in \Sigma$ there are at most two distinct curvature spheres. 
Umbilic points are exactly the points where there is only one 
curvature sphere. In that case $T_{s(x)}= T_{x}\Sigma$. 
In the case that $\mathcal{F}$ is umbilic-free, the curvature 
spheres form two rank one subbundles $s_{1}$ and $s_{2}$ of $\mathcal{F}$ 
(called {\it curvature sphere congruences}\/) 
with respective rank one curvature subbundles 
$T_{1}:= \bigcup_{x\in \Sigma} T_{s_{1}(x)}$ and 
$T_{2}:= \bigcup_{x\in \Sigma} T_{s_{2}(x)}$ of the tangent 
bundle $T\Sigma$. Furthermore, 
$\mathcal{F}=s_{1}\oplus s_{2}$ and $T\Sigma = T_{1}\oplus T_{2}$.

Suppose now that we have a point sphere complex $p$ and a spaceform 
vector $q$ defining $M$ and $N$ as in~(\ref{eqn:spforms}). 
Let $F=\mathcal{F}\cap M$ denote the spaceform 
projection and $T=\mathcal{F}\cap N$ denote the tangent plane congruence 
of $\mathcal{F}$. Where $F$ is immersed, let $\kappa_{1}$ and 
$\kappa_{2}$ denote 
the principal curvatures of $F$. 
Then Cecil~\cite[Chapter 4]{cecil} showed that $T+\kappa_{i}F$ are 
lifts of the curvature sphere congruences $s_{i}$ of $\mathcal{F}$. 
This can be deduced by using Rodrigues equations in the spaceform $M$. 
For example, consider the lifts $F$ and $T$ of an umbilic-free surface 
$f:\Sigma\to \mathbb{R}^{3}$ in Euclidean space with unit normal 
$t:\Sigma\to S^{2}$ given in~(\ref{eq:lifts}). 
For $i\in\{1,2\}$ let $\partial_i$ be a
principal curvature direction in $T\Sigma$ 
corresponding to $\kappa_{i}$. Then by Rodrigues equations 
$\partial_{i}t+\kappa_{i}\partial_{i}f=0$. 
One can then check that this implies 
$\partial_i T + \kappa_i \partial_i F =0 $. Hence, 
\[ \partial_{i}(T+\kappa_{i}F) 
= \partial_{i}T+\partial_{i}\kappa_{i}\,F+\kappa_{i}\partial_{i}F 
= \partial_{i}\kappa_{i}\,F\in \Gamma \mathcal{F}.\]
Thus, one has that $s_{i}=\span\{T+\kappa_{i}F\}$ and 
$T_{i}=\span\{\partial_{i}\}$. 

It is clear from the lifts $T+\kappa_{i}F$ that umbilic 
points of $F$, i.e., points where the principal curvatures coincide, 
are exactly the umbilic points of $\mathcal{F}$, i.e., 
the points where $\mathcal{F}$ has only one curvature sphere.

It should come as no surprise that the sphere 
$S_{i}:=M\cap s_{i}^{\perp}$ corresponding to $s_{i}$ coincides 
with the classical notion of a curvature sphere, i.e., 
a sphere tangent to the surface $F$ with radius $\kappa_{i}^{-1}$. 
Therefore, for any point $x\in \Sigma$, $S_{i}(x)$ has second order 
contact with $F$ in the principal curvature direction $T_{i}(x)$. 
\subsection{M\"obius and Lie sphere transformations}

Let $\hat M$ denote $M$ union its 
(possible empty) ideal boundary $\partial M$.  
M\"obius transformations of $M$ are the diffeomorphisms from $\hat M$ 
to $\hat M$ that map spheres to 
spheres, and they 
are represented by pseudo-orthogonal transformations of $\mathbb{R}^{4,2}$ 
(i.e. transformations that preserve 
$\inner{~}{~}$, and 
$O(4,2)$ is a double cover of them) 
that fix $p$.  
M\"obius transformations of $M$ 
preserve the conformal structure, so will preserve the conformal 
structure of 
any surface inside $M$ as well.  Furthermore, 
M\"obius transformations 
will preserve contact orders of any spheres tangent to the 
surface, and so will preserve the principal curvature 
spheres. As a direct consequence, an 
umbilic point of the 
surface will remain an umbilic point 
after the M\"obius transformation is applied.

Lie sphere transformations are the transformations of spaceforms
that map spheres to spheres and preserve the oriented 
contact of spheres, and, like for 
M\"obius transformations, they are 
represented by matrices in 
$O(4,2)$, and 
the group of Lie sphere transformations is 
now isomorphic to all of 
$O(4,2)/\{ \pm I \}$.  M\"obius 
transformations are of course 
a special case of this.  The objects 
preserved by general Lie sphere 
transformations are oriented spheres, but 
not point spheres (unlike M\"obius 
transformations).  
However, contact elements are mapped to contact elements, 
so restricting to point spheres within the 
contact elements gives maps 
taking points to points.  From this latter 
point of view, all Lie sphere transformations 
of surfaces in $M$ are generated (by composition) 
from M\"obius transformations and parallel surface 
transformations in various spaceforms 
(see \cite{cecil}, Theorem 3.18).  
Like for M\"obius transformations, 
curvature spheres are preserved under Lie sphere transformations.

Rephrasing the above statements 
about Lie sphere transformations 
more precisely, 
given a Legendre immersion $\mathcal{F}:\Sigma\to \mathcal{Z}$, 
we have that $A\mathcal{F}$ is a Legendre immersion, for any 
$A\in O(4,2)$. 
Furthermore, if $s(x)$ is a curvature sphere of $\mathcal{F}$ at $x$ 
with curvature space $T_{s(x)}$, then $As(x)$ is a curvature sphere 
of $A\mathcal{F}$ at $x$ with curvature space $T_{As(x)}=T_{s(x)}$.  
With $p,q$ as chosen 
just before \eqref{eq:lifts}, 
let $f=(f_1,f_2,f_3):\Sigma\to \R^3$ be the 
regular surface in $\R^3$ with unit normal 
$t:\Sigma\to S^{2}$ given by $\mathcal{F}$.  
Continue on from \eqref{eq:assump0} and \eqref{eq:projeceuc},
assuming that for a 
generator $\{\sigma_1,\sigma_2\}$ of ${\cal F}$ 
(we could try for 
$\sigma_1=F,\sigma_2=T:\Sigma\to \R^{4,2}$ as defined 
in \eqref{eq:lifts}, for example) the matrix 
$$
B_A
=
\pmt{
\inner{A\sigma_1}{p}&\inner{A\sigma_2}{p}\\
\inner{A\sigma_1}{q}&\inner{A\sigma_2}{q}}\quad
\text{is regular},
$$
then
\begin{equation}\label{eq:liesptrdef}
\begin{array}{l}
\hat{F}
=
a A\sigma_1+ b A\sigma_2\\
\hat{T}
=
c A\sigma_1+ d A\sigma_2,
\end{array}
\quad
\pmt{a&c\\ b&d}
=
B_A^{-1}\pmt{0&-1\\-1&0}
\end{equation}
satisfy
$\inner{\hat{F}}{p}=0$,
$\inner{\hat{F}}{q}=-1$,
$\inner{\hat{T}}{p}=-1$ and
$\inner{\hat{T}}{q}=1$.
Thus $\hat{F}$, $\hat{T}$ have the form
\begin{equation}\label{eq:liesptrdefeuc}
\hat{F}=\trans{\left(
\dfrac{1}{2}\left(1+\innera{\hat f}{\hat f}\right),
\dfrac{1}{2}\left(1-\innera{\hat f}{\hat f}\right),
\hat f,0\right)},\quad
\hat{T}=
\trans{\left(\innera{\hat f}{\hat t},
-\innera{\hat f}{\hat t},\hat t,1\right)}.
\end{equation}
Hence we can
project $A{\cal F}=\{A\sigma_1,A\sigma_2\}$ to $M$ (isometric to $\R^3$) 
and $N$, and we obtain 
$
\hat f :\Sigma\to \R^3
$
and $
\hat t :\Sigma\to S^2
$.
We call $\hat{f}$ the {\it Lie sphere transformation of\/ $f$ by\/ $A$}.
The unit normal of $\hat{f}$ is $\hat{t}$.
\section{Conditions for singularities}
\subsection{Criteria for singularities}
Let $\Sigma \subset \R^2$ be an open domain and
$f:\Sigma \to\R^3$ a {\em frontal}, 
meaning there exists a unit normal vector field $t:\Sigma \to S^2$.
The function 
$
\det(f_u,f_v,t)
$
is called the {\em signed area density function\/}
with respect to the local coordinate system
$(\Sigma ;u,v)$ of $\Sigma$,
where $f_*=\partial f/\partial *$ for $*=u,v$. 
If the map $(f,t):\Sigma \to\R^3\times S^2$ is
an immersion, $f$ is called a {\em front}.
Let $S(f)$ be the set of singular points of $f$.
Take a point $x \in S(f)$ such that $\rank df_x=1$.
Then there exists a non-zero
vector field $X$ on a small neighborhood $V$ of $x$
such that $df_y(X)=0$ holds for 
$y\in S(f) \cap V$.
We call $X$ a {\em null vector field\/}
with respect to $f$.

Let $g_1,g_2:(\R^2,\zv)\to(\R^3,\zv)$
be two map-germs at the origin.
These map-germs are {\em ${\cal A}$-equivalent\/}
if there exist diffeomorphism-germs
$\Xi_s:(\R^2,\zv)\to(\R^2,\zv)$ and
$\Xi_t:(\R^3,\zv)\to(\R^3,\zv)$
satisfying $g_2\circ \Xi_s=\Xi_t\circ g_1$.
A map-germ at the origin $g:(\R^2,\zv)\to(\R^3,\zv)$
is a cuspidal edge, swallowtail, cuspidal beaks, cuspidal lips, 
cuspidal butterfly or $D_4^\pm$ singularity, respectively, 
if it is ${\cal A}$-equivalent to the map germ 
\begin{itemize}
\item $(u,v)\mapsto(u,v^2,v^3)$ (cuspidal edge),
\item $(u,v)\mapsto(u,4v^3+2uv,3v^4+uv^2)$ (swallowtail),
\item $(u,v)\mapsto(u,2v^3-u^2v,3v^4-u^2v^2)$ (cuspidal beaks),
\item $(u,v)\mapsto(u,2v^3+u^2v,3v^4+u^2v^2)$ (cuspidal lips),
\item $(u,v)\mapsto(u,5v^4+2uv,4v^5+uv^2)$ (cuspidal butterfly),
\item $(u,v)\mapsto(2 u v,\pm u^2+3 v^2, \pm 2 u^2 v+2 v^3)$ 
($D_4^\pm$ singularity)
\end{itemize}
at the origin.

We have the following well-known characterizations of
cuspidal edges and swallowtails \cite[Proposition 1.3]{krsuy}, 
see also \cite[Corollary 2.5]{suy3}.
Also there are characterizations of
cuspidal lips, cuspidal beaks and
cuspidal butterflies, 
see \cite[Theorem A.1]{horoflat} and 
\cite[Theorem 8.2]{mandala}.
Two function-germs, or map-germs, 
are {\em proportional\/} if they 
coincide up to non-zero scalar functional 
multiplication.
\begin{lemma}\label{lem:crit}
Let\/ $f$ be a front and\/ $x$ a singular point
for which\/ $df_x$ is rank\/ $1$.
Let\/ $\lambda$ be a function which is proportional
to the signed area density function, 
and\/ $X$ a null vector field in a neighborhood of\/ $x$.
Then\/ $f$ at\/ $x$ is a 
\begin{enumerate}
\item cuspidal edge if and only if\/
$d_{X}\lambda\ne0$ at\/ $x$,
\item swallowtail if and only if\/
$d_{X}\lambda=0$, $d_{X}d_{X}\lambda\ne0$ and\/
$d\lambda\ne0$ at\/ $x$.
\item cuspidal beaks if and only if\/
$d\lambda=0$, $d_{X}d_{X}\lambda\ne0$ and\/
$\det\Hess\lambda<0$ at\/ $x$,
\item cuspidal lips if and only if\/
$d\lambda=0$ and\/
$\det\Hess\lambda>0$ at\/ $x$,
\item cuspidal butterfly if and only if\/
$d_{X}\lambda=d_{X}d_{X}\lambda=0$, 
$d_{X}d_{X}d_{X}\lambda\ne0$ and\/
$d\lambda\ne0$ at\/ $x$.
\end{enumerate}
\end{lemma}
\red{We shall now define three types of singularity. 
The motivation for defining these types 
will become clear in Section~\ref{sec:singliesp}, 
where we apply Lemma~\ref{lem:crit} to the 
study of Lie sphere transformations of surfaces.
Let $f$ be a front and $x$ a singular point such that $\rank df=1$.
Let $\lambda$ be a function which is proportional
to the signed area density function, 
and let $X$ be a null vector field in a neighborhood of $x$.
Then 
$x$ is called a {\it type\/ $1$ singularity\/}
if $d_X\lambda\ne0$.
Furthermore, 
$x$ is called a {\it type\/ $2$ singularity\/}
if $d_X\lambda=0$ and $d_Xd_X\lambda\ne0$ at $x$,
and is called a {\it type\/ $3$ singularity\/}
if $d_X\lambda=d_Xd_X\lambda=0$ and 
$d_Xd_Xd_X\lambda\ne0$ at $x$.
By the non-degeneracy, one can show that
these definitions do not depend on the choice
of $\lambda$ and $X$.
By Lemma \ref{lem:crit}, if $x$ is a type $1$ singularity,
then $x$ is a cuspidal edge.
If $x$ is a non-degenerate type $2$ (respectively, type $3$)
singularity, then $x$
is a swallowtail (respectively, cuspidal butterfly). 
Cuspidal beaks and cuspidal lips 
are examples of
type $2$ degenerate singularities.
}

In the case that $\rank df_x=0$,
there is the following characterization for 
$D_4^\pm$ singularities: 
\begin{lemma}\label{lem:d4} {\rm (\cite{d4})}
Let\/ $f$ be a front with unit normal\/ $t$ and 
let\/ $\lambda$ be a function which is proportional to the
signed area density function.  
A singular point\/ $x$ is a \/ $D_4^+$ 
$($respectively, $D_4^-)$ singularity 
if and only if 
the following two conditions hold:
\begin{enumerate}
\item $\rank df_x=0$.
\item $\det\Hess \lambda<0$ 
$($respectively, $\det\Hess \lambda>0)$ at\/ $x$.
\end{enumerate}
\end{lemma}

Lemma~\ref{lem:crit} and Lemma~\ref{lem:d4} show how certain types 
of singularities are determined by any function proportional 
to the signed area density function. We will now derive such a 
function in the context of Lie sphere geometry. 
Using the 
lifts of (\ref{eq:lifts}), 
for any linearly independent vector
fields $X,Y\in \Gamma T\Sigma$, 
the determinant $\det(d_{X}F, d_{Y}F , T, F, q, p)$
is equal to the determinant of
\[
\begin{pmatrix} 
\innera{d_Xf}{f} & \innera{d_Yf}{f} & \innera{f}{t}  
&\frac{1}{2}(1+\innera{f}{f}) & 1 & 0\\
-\innera{d_Xf}{f} & -\innera{d_Yf}{f} & -\innera{f}{t} 
 &\frac{1}{2}(1-\innera{f}{f}) & -1 & 0\\
d_{X}f & d_{Y}f& t & f & 0&0\\
0 & 0 & 1 & 0 & 0 & 1
\end{pmatrix}.
\]
This is the same as 
\[  -\det(d_{X}f,d_{Y}f,t) \det \begin{pmatrix} 
\frac{1}{2}(1+\innera{f}{f})  & 1 & 0\\
\frac{1}{2}(1-\innera{f}{f}) & -1 & 0\\
0 & 0 & 1
\end{pmatrix}
= \det(d_{X}f,d_{Y}f,t).\]
Hence, the signed area density function 
can be taken to be 
\[ \lambda = \det (d_{X}F, d_{Y}F , T, F, q, p).\]
Now given two linearly independent sections 
$\sigma, \tilde{\sigma}\in \Gamma \mathcal{F}$, we may write 
\[ \sigma = -(\sigma,q)F -(\sigma,p)T \quad 
\text{and}\quad \tilde{\sigma} = -(\tilde{\sigma},q)F -(\tilde{\sigma},p)T.\]
Thus, 
\[ F = \frac{1}{\Delta}\big( (\sigma,p)\tilde{\sigma} 
- (\tilde{\sigma},p)\sigma\big),\]
where 
$\Delta:= (\tilde{\sigma},p)(\sigma,q) 
- (\sigma,p)(\tilde{\sigma},q)$. Hence, 
\[ dF = 
\frac{1}{\Delta}\big( (\sigma,p)\,d\tilde{\sigma} 
- (\tilde{\sigma},p)\,d\sigma\big)\, 
\ \operatorname{mod}\ \Omega^{1}(\mathcal{F})\]
and the signed area density function is proportional to 
\begin{equation} 
\label{eq:areafcn}
\det \Big((\sigma,p)d_{X}\tilde{\sigma} - (\tilde{\sigma},p)
d_{X}\sigma,\ (\sigma,p)d_{Y}\tilde{\sigma} 
- (\tilde{\sigma},p)d_{Y}\sigma,\ T,\ F,\ q,\ p\Big).
\end{equation}

\subsection{Calculations at a non-umbilic point}
Let $\mathcal{F}$ be a Legendre immersion,
and $F$ as in \eqref{eq:lifts}.
In this section,
we assume that $x$ is a non-umbilic point of $f$. 
Then, as stated in the introduction,  on a neighborhood of $x$ there 
are two distinct 
curvature sphere congruences $s_{1}$ and $s_{2}$ with 
curvature subbundles $T_{1}$ and $T_{2}$, respectively. 
We may then choose $\tilde{\sigma} = \sigma_{1}$ and 
$\sigma=\sigma_{2}$ in 
(\ref{eq:areafcn}) for any non-zero lifts 
$\sigma_{1}\in \Gamma s_{1}$ and 
$\sigma_{2}\in\Gamma s_{2}$. Furthermore, we take 
$X\in \Gamma T_{1}$ and $Y\in \Gamma T_{2}$. 
Then since 
$d_{X}\sigma_{1}, d_{Y}\sigma_{2}\in \Gamma \mathcal{F}$ holds,
(\ref{eq:areafcn}) becomes 
\begin{equation}
\label{eqn:detsgn}
 - (\sigma_{1},p)(\sigma_{2},p) \det 
( d_{X}\sigma_{2},d_{Y}\sigma_{1} , T, F, q, p).
\end{equation}
This gives rise to the following lemma:

\begin{lemma}
\label{lem:signprop}
The signed area density function\/ 
$\lambda$ 
of\/ $f$ as in\/ \eqref{eq:lifts}
is proportional to\/ $(\sigma_{1},p)(\sigma_{2},p)$ 
for any choice of non-zero lifts of the curvature spheres\/ 
$\sigma_{1}\in \Gamma s_{1}$ and\/ $\sigma_{2}\in \Gamma s_{2}$. 
\end{lemma}
\begin{proof}
Following~(\ref{eqn:detsgn}), it is sufficient to show that  
$\{d_{X}\sigma_{2}, d_{Y}\sigma_{1}, T,F,q,p\}$ is a basis 
for $\Sigma\times \R^{4,2}$. Since $\mathcal{F}$ is an isotropic map, 
we have that $d_{X}\sigma_{2}, d_{Y}\sigma_{1}\in \Gamma\mathcal{F}^{\perp}$.
Now if there exist functions $\lambda$ and $\mu$ such that 
$\lambda d_{X}\sigma_{2}+ \mu d_{Y}\sigma_{1}\in \Gamma \mathcal{F}$, 
then $d_{\lambda X + \mu Y}(\sigma_{2}+\sigma_{1})\in \Gamma \mathcal{F}$. 
However, this implies that $\lambda = \mu =0$ as otherwise 
$\sigma_{2}+\sigma_{1}$ would span a curvature sphere congruence, 
contradicting that there are exactly two curvature sphere congruences 
$s_{1}$ and $s_{2}$. 
Thus, $\span\{d_{X}\sigma_{2}, d_{Y}\sigma_{1}\}\oplus \mathcal{F}$ 
is a rank $4$ subbundle of $\mathcal{F}^{\perp}$, i.e., 
\[  \span\{d_{X}\sigma_{2}, d_{Y}\sigma_{1}, T,F\} = \mathcal{F}^{\perp}.\]
By~(\ref{eq:assump0}), we have that 
$\span\{q,p\}\cap \mathcal{F}^{\perp} = \{0\}$ and the result follows. 
\end{proof}

It is clear from Lemma~\ref{lem:signprop} that if $\lambda(x)=0$ at 
a point $x\in \Sigma$, then $s_{1}(x)\perp p$ or $s_{2}(x)\perp p$. 
Note that since $p$ is timelike, we cannot have that both 
$s_{1}(x)\perp p$ and $s_{2}(x)\perp p$. 
Assume, without loss of generality,  that $s_{1}(x)\perp p$. 
Then by inertia, there exists an open neighborhood $V$ of $x$ 
such that $s_{2}(y)\not\perp p$, for all $y\in V$. Therefore, 
$y\in S(f)\cap V$ if and only if $s_{1}(y)\perp p$. 
This is equivalent to $F(y)\in s_{1}(y)$. In such an instance, 
since $d_{X_{y}}\sigma_{1}\in \mathcal{F}(y)$ for all 
$X\in \Gamma T_{1}$ and $\sigma_{1}\in \Gamma s_{1}$, 
we have that $d_{X_{y}}F=0$. Hence $d_{X_{y}}f=0$ and any 
$X\in \Gamma T_{1}$ locally yields a null vector field with respect to 
$f$.  We conclude that: 

\begin{lemma}
The point\/
$x$ is a singular point of\/ $f$ if and only if\/ 
$s_{1}(x)\perp p$ or $s_{2}(x)\perp p$. 
In which case, any\/ $X\in \Gamma T_{1}$ or\/ 
$Y\in \Gamma T_{2}$, respectively, 
locally yields a null vector field for\/ $f$. 
\end{lemma}

{\bf Conditions for singularities.}
We will now state and prove our main theorem:
\begin{theorem}
\label{thm:singcond}
Let\/ $x\in \Sigma$ be a non-umbilic point of\/ $\mathcal{F}$ 
and let\/ $p$ be a point sphere complex. 
Suppose that\/ $s_{1}(x)$ is perpendicular to\/ $p$
$($and thus\/ $\mathcal{F}$ projects to a singular point
in any spaceform with point sphere complex\/ $p)$.
Let\/ $\sigma_{1}\in \Gamma s_{1}$ be any non-zero 
lift of\/ $s_{1}$, and let\/ $X\in \Gamma T_{1}$. Then 
\begin{enumerate}
\item\label{item:edge} $d_{X_{x}}\sigma_{1}\not\in s_{1}(x)$ 
if and only if\/ $\mathcal{F}$ 
projects to a cuspidal edge at\/ $x$, 

\item\label{item:slb} $d_{X_{x}}\sigma_{1}\in s_{1}(x)$,
$(d_{X}d_{X}\sigma_{1})_{x}\not\in s_{1}(x)$
\red{
if and only if\/ $\mathcal{F}$ projects to a type\/ $2$ singularity
at\/ $x$}.
Moreover, under this condition,
\begin{itemize}
\item $d\sigma_{1}\not\perp p$ at\/ $x$, 
if and only if\/ $\mathcal{F}$ projects to a swallowtail at\/ $x$,

\item $d\sigma_{1}\perp p$, and\/
$\det\Hess (\sigma_1,p) >0$ $($respectively, $\det\Hess (\sigma_1,p) <0)$
at\/ $x$ if and only if\/ 
$\mathcal{F}$ projects to a cuspidal lips 
$($respectively, cuspidal beaks\/$)$
at\/ $x$,
\end{itemize}
\item \label{item:cuspbutt}
$d_{X_{x}}\sigma_{1}, (d_{X}d_{X}\sigma_{1})_{x}\in s_{1}(x)$, 
$(d_{X}d_{X}d_{X}\sigma_{1})_{x}\not\in s_{1}(x)$
\red{
if and only if\/ $\mathcal{F}$ projects to a type\/ $3$ singularity\/ $x$}.
Moreover, under this condition,
$d\sigma_{1}\not\perp p$ at\/ $x$ if and only if\/ 
$\mathcal{F}$ projects to a cuspidal butterfly at\/ $x$,
\end{enumerate}
\end{theorem}

\begin{proof}
In this proof we will prove a series of facts equating conditions on 
$\lambda$ to conditions involving $s_{1}$ and the point sphere complex $p$. 
It is then straightforward to complete this proof by applying these 
facts to Lemma~\ref{lem:crit}. 

Let $\sigma_{2}$ be a non-zero lift of the curvature sphere $s_{2}$. 
Since $\sigma_{1}$ is a lift of the curvature sphere congruence $s_{1}$ 
and $X\in \Gamma T_{1}$, we have that 
\begin{equation} 
\label{eqn:alphabeta}
d_{X}\sigma_{1} = \alpha \sigma_{1} + \beta \sigma_{2},
\end{equation}
for some smooth functions $\alpha$ and $\beta$. We prove this theorem 
by utilizing the result of Lemma~\ref{lem:signprop} that we may replace 
the signed area density function in Lemma~\ref{lem:crit} by 
\[ \lambda = (\sigma_{1},p)(\sigma_{2},p),\]
for any non-zero lift of $\sigma_{2}\in \Gamma s_{2}$. 
By the Leibniz rule, 
we then have for any $V\in \Gamma T\Sigma$,  
\begin{equation} 
\label{eqn:dsign}
d_{V}\lambda = (d_{V}\sigma_{1},p)(\sigma_{2},p) 
+ (\sigma_{1},p)(d_{V}\sigma_{2},p).
\end{equation}
Since we assumed that $s_{1}(x)$ is perpendicular to $p$, we then have that 
\[ d_{V_{x}}\lambda = (d_{V_{x}}\sigma_{1},p)(\sigma_{2}(x),p).\] 
By the assumption in~(\ref{eq:assump0}) we have that 
$(\sigma_{2}(x),p)\neq 0$, thus $d_{V_{x}}\lambda =0$ 
if and only if $(d_{V_{x}}\sigma_{1},p)=0$. 
Hence, 
\begin{equation}
\label{eqn:dlambda}
\text{$d\lambda =0$ at $x$ if and only if $(d\sigma_{1},p)=0$ at $x$}.
\end{equation}
If we replace $V$ with $X$ then by~(\ref{eqn:alphabeta}), 
$(d_{X_{x}}\sigma_{1},p)=0$ if and only if $\beta(x)=0$. 
Therefore, we have shown that
\begin{equation}
\label{eqn:dXlambda}
\text{
$d_{X}\lambda = 0 $ at $x$ if and only if $d_{X_{x}}\sigma \in s_{1}(x)$.}
\end{equation} 

Differentiating~(\ref{eqn:dsign}) with respect to $W\in \Gamma T\Sigma$, 
we have that $d_{W}d_{V}\lambda$ is given by 
\begin{equation}
\label{eqn:ddsign}
(d_{W}d_{V}\sigma_{1},p)(\sigma_{2},p) 
+
 (d_{V}\sigma_{1},p)(d_{W}\sigma_{2},p)
+
 (d_{W}\sigma_{1},p)(d_{V}\sigma_{2},p) 
+ (\sigma_{1},p)(d_{W}d_{V}\sigma_{2},p).
\end{equation}
If we assume that $d_{X_{x}}\sigma_{1}\in s_{1}(x)$ then 
$(d_{X_{x}}\sigma_{1},p) = (\sigma_{1},p)=0$ and thus 
\[ (d_{X}d_{X}\lambda)_{x} 
= \big((d_{X}d_{X}\sigma_{1})_{x},p\big)(\sigma_{2},p).\]
Hence, $(d_{X}d_{X}\lambda)_{x}=0$ if and only if 
$\big((d_{X}d_{X}\sigma_{1})_{x},p\big)=0$. On the other hand, 
by~\eqref{eqn:alphabeta}, we have that 
\[ (d_{X}d_{X}\sigma_{1})_{x} = \big(d_{X_{x}}\alpha + \alpha^{2}(x)\big)\, 
\sigma_{1} + (d_{X_{x}}\beta)\,\sigma_{2}, \]
since $\beta(x)=0$. Therefore, $\big((d_{X}d_{X}\sigma_{1})_{x},p\big)=0$ 
if and only if $d_{X_{x}}\beta=0$, or equivalently, 
$(d_{X}d_{X}\sigma_{1})_{x}\in s_{1}(x)$. 
To summarize, we have shown that if $d_{X}\lambda =0$ at $x$ then 
\begin{equation}
\label{eqn:dXdXlambda}
 \text{$d_{X}d_{X}\lambda=0$ at $x$ if and only if 
$(d_{X}d_{X}\sigma_{1})_x \in s_{1}(x)$.}
\end{equation}
If we assume that $d\sigma_{1}\perp p$ at $x$, then by~(\ref{eqn:ddsign}),  
\[ (d_{W}d_{V}\lambda)_{x} 
= \big((d_{W}d_{V}\sigma_{1})_{x}, p\big)(\sigma_{2}(x),p).\]
Hence, 
\begin{equation}
\label{eqn:hesslambda}
(\det\Hess\lambda)_{x} = (\det\Hess(\sigma_{1},p))_{x} (\sigma_{2}(x),p)^{2},
\end{equation}
and thus $\det\Hess\lambda>0$ (respectively, $\det\Hess\lambda<0$) 
at $x$ if and only if $\det\Hess(\sigma_{1},p)>0$ (respectively, 
$\det\Hess(\sigma_{1},p)<0$) at $x$. 

By differentiating~(\ref{eqn:ddsign}) and using similar arguments as 
for the other cases, it is straightforward to prove that 
if $d_{X}\lambda =0$ and $d_{X}d_{X}\lambda=0$ at $x$ then 
\begin{equation}
\label{eqn:dXdXdXlambda}
\text{$d_{X}d_{X}d_{X}\lambda=0$ at $x$ if and only if 
$(d_{X}d_{X}d_{X}\sigma_{1})_x 
\in s_{1}(x)$.}
\end{equation} 

It is now straightforward to complete the proof of 
Theorem~\ref{thm:singcond} by applying the facts (\ref{eqn:dlambda}), 
(\ref{eqn:dXlambda}), (\ref{eqn:dXdXlambda}), (\ref{eqn:hesslambda}) 
and (\ref{eqn:dXdXdXlambda}) appropriately to Lemma~\ref{lem:crit}. 
\end{proof}

%We refer to 
%cuspidal edges as {\it class\/ $1$ singularities\/};
%swallowtails, cuspidal lips, cuspidal beaks and 
%degenerate singularities of type $2$ 
%as {\it class\/ $2$ singularities\/};
%and
%cuspidal butterflies and degenerate singularities of type $3$ 
%as {\it class\/ $3$ singularities\/}.

%\begin{figure}[phbt]
%\label{firstfigure}
%\begin{center}
%\begin{tabular}{ccc}
%\includegraphics[width=0.4\linewidth]{figs/firstfig-swallowtail.eps} &
%\includegraphics[width=0.24\linewidth]{figs/firstfig-beaks.eps} &
%\end{tabular}
%\begin{tabular}{ccc}
%\includegraphics[width=0.3\linewidth]{figs/firstfig-lips.eps} &
%\includegraphics[width=0.3\linewidth]{figs/firstfig-degenerate.eps} &
%\end{tabular}
%\caption{A swallowtail, cuspidal beaks ($\xi < 1/(2 \sqrt{2})$), 
%cuspidal lips %($\xi > 1/(2 \sqrt{2})$), more 
%degenerate singularity ($\xi = 1/(2 \sqrt{2})$) as given in 
%Example \ref{firste%xample}.}
%\end{center}
%\end{figure}

\subsection{Calculations at an umbilic point}

Suppose that $x$ is an umbilic point of $\mathcal{F}$, i.e., 
there exists $s(x)$ in 
$\mathcal{F}(x)$ such that for any section 
$\sigma\in \Gamma \mathcal{F}$ with $\sigma(x)\in s(x)$ we have that 
\[ (d\sigma)_{x} \in T_{x}\Sigma \otimes \mathcal{F}(x).\]
 Define 
\begin{equation}
\label{eqn:cubform}
\mathcal{C}(X,Y,Z)(\sigma,\tilde{\sigma}):= 
(d_{X}d_{Y}d_{Z}\tilde{\sigma},\sigma),
\end{equation}
for $X,Y,Z\in \Gamma T\Sigma$ and 
$\sigma, \tilde{\sigma}\in \Gamma \mathcal{F}$ such that 
$\sigma(x)\in s(x)$. 

\begin{lemma}
The map\/
$\mathcal{C}$ is tensorial at\/ $x$ and 
is symmetric in\/ $X,Y,Z$ at\/ $x$. 
Furthermore, $\mathcal{C}_{x}(s(x),s(x))=0$. 
Hence, we may identify\/ $\mathcal{C}_{x}$ as an element of\/ 
$S^{3}(T^{*}_{x}\Sigma)\otimes s(x)\otimes f(x)/s(x)$.
\end{lemma}
\begin{proof}
Firstly, let $Y,Z\in \Gamma T\Sigma$. Then, using the Leibniz rule 
and that $\mathcal{F}$ is isotropic,
\[ (d_{Y}d_{Z}\tilde{\sigma},\sigma) 
= - (d_{Z}\tilde{\sigma},d_{Y}\sigma) 
=  (\tilde{\sigma},d_{Z}d_{Y}\sigma) = (\tilde{\sigma},d_{Y}d_{Z}\sigma) .
\]
Since $\sigma(x)\in s(x)$, 
$(d\sigma)_{x}\in T_{x}\Sigma\otimes \mathcal{F}(x)$ and thus 
$(d_{Z}\tilde{\sigma},d_{Y}\sigma)=0$ at $x$. Hence, 
\begin{equation}
\label{eqn:dydz}
((d_{Y}d_{Z}\tilde{\sigma})_{x},\sigma(x)) = 
(\tilde{\sigma}(x),(d_{Y}d_{Z}\sigma)_{x}) = 0.
\end{equation}
In particular, since $\tilde{\sigma}$ is arbitrary, 
we have that $(d_{Y}d_{Z}\sigma)_{x}\in (\mathcal{F}(x))^{\perp}$. 
Using the Leibniz rule and the fact that $\mathcal{F}$ 
is isotropic one can show that 
\begin{equation}\label{eq:dxdydzsigma}
(d_{X}d_{Y}d_{Z}\sigma,\sigma) = -(d_{X}d_{Y}\sigma,d_{Z}\sigma)
- (d_{Y}\sigma,d_{X}d_{Z}\sigma) - (d_{Y}d_{Z}\sigma,d_{X}\sigma).
\end{equation}
Therefore, since $(d\sigma)_{x}\in T_{x}\Sigma\otimes \mathcal{F}(x)$,
\begin{equation}
\label{eqn:dxdydz}
\big((d_{X}d_{Y}d_{Z}\sigma)_{x},\sigma(x)\big)=0.
\end{equation}
Now let us assume that $\sigma$ and $\tilde{\sigma}$ are everywhere 
linearly independent. Then an arbitrary section 
$\hat{\sigma}\in \Gamma \mathcal{F}$ may be written as 
$\hat{\sigma} = \alpha \tilde{\sigma}+\beta \sigma$ for some 
smooth functions $\alpha$ and $\beta$. By applying~(\ref{eqn:dydz}) 
and~(\ref{eqn:dxdydz}), we have that 
\[ \big((d_{X}d_{Y}d_{Z}\hat{\sigma})_{x}, \sigma(x)\big) 
= \alpha(x) \big((d_{X}d_{Y}d_{Z}\tilde{\sigma})_{x},\sigma(x)\big).\]
Hence, the value of $\mathcal{C}$ at $x$ only depends on 
$\sigma(x)$ and $\tilde{\sigma}(x)\, \operatorname{mod}\, s(x)$.

The tensorality of $\mathcal{C}$ in $X,Y,Z$ at 
$x$ follows from the linearity of 
the connection $d$ and~(\ref{eqn:dydz}). 
The symmetry of $\mathcal{C}$ at $x$ follows 
from~(\ref{eqn:dydz}) and the flatness of the trivial connection $d$. 
\end{proof}

Suppose that $c\in S^{3}(T_{x}^{*}\Sigma)$ is a cubic form. 
Then in terms of a basis $X,Y\in T_{x}\Sigma$, its discriminant is 
given by
\[ (c_{X,X,X}c_{X,Y,Y}- c_{X,Y,X}^{2})
(c_{Y,Y,Y}c_{X,Y,X}- c_{X,Y,Y}^{2}) 
- (c_{X,X,X}c_{Y,Y,Y}-c_{X,Y,X}c_{X,Y,Y})^{2}.\]
Now since $s(x)\otimes f(x)/s(x)$ is a line bundle, 
$\mathcal{C}_{x}$ may be viewed as a conformal class of cubic forms. 
Since this is a conformal class, the sign of the 
discriminant of any two non-zero elements coincides. 
We will refer to this as the \textit{sign of 
the discriminant of\/ $\mathcal{C}_{x}$}. 

Now suppose we are given a point sphere complex $p$ 
and spaceform 
vector $q$. 
Let $F:\Sigma\to M$ denote the spaceform projection 
and $T:\Sigma\to N$ denote the tangent plane 
congruence of $\mathcal{F}$. 
If $f$ is immersed at $x$ with principal 
curvature $\kappa=\kappa_1(x)=\kappa_2(x)\in\R$, 
then  Rodrigues' equations we have that 
$\sigma:= T+\kappa F$ satisfies 
$(d\sigma)_{x} \in T_{x}\Sigma\otimes \mathcal{F}(x)$. 
Let $\tilde{\sigma}:= F$. 
Then in terms of this choice of $\sigma$ 
and $\tilde{\sigma}$ we have that 
\begin{equation*}
\begin{array}{rcl}
\mathcal{C}_{x}(X,Y,Z)(T+\kappa F,F) 
&=&  (d_{X}d_{Y}d_{Z}F, T+\kappa F)\\
&=& (d_{X}d_{Y}d_{Z}F, T) +\kappa (d_{X}d_{Y}d_{Z}F, F),\\
&=& (d_{X}d_{Y}d_{Z}F, T) \\
&&-\kappa 
\big((d_{X}d_{Y}F,d_{Z}F) + (d_{X}d_{Z}F,d_{Y}F) + 
(d_{Y}d_{Z}F,d_{X}F)\big)
\end{array}
\end{equation*}
by \eqref{eq:dxdydzsigma},
noticing that $\kappa$ is a constant.
This coincides with the cubic form in Porteous \cite{porteous}. 
Hence, we obtain the following result:

\begin{proposition}\label{prop:umbilic}
Suppose that the sign of the discriminant of\/ $\mathcal{C}_{x}$ 
is positive\/ $($respec\-tively, negative\/$)$. 
Then, when the spaceform projection immerses, 
$\mathcal{F}$ projects to a elliptic\/ $($hyperbolic\/$)$ 
umbilic at\/ $x$. 
\end{proposition} 

We will now examine the case that the spaceform projection does 
not immerse at $x$. Assume that $x\in \Sigma$ is an umbilic 
point of $\mathcal{F}$. Recall from~(\ref{eq:areafcn}) that 
given two linearly independent sections 
$\sigma, \tilde{\sigma}\in \Gamma \mathcal{F}$, 
we may assume that the signed area density function is proportional to 
\[ \det \Big(
(\sigma,p)d_{X}\tilde{\sigma} - (\tilde{\sigma},p)d_{X}\sigma,\ 
(\sigma,p)d_{Y}\tilde{\sigma} - (\tilde{\sigma},p)d_{Y}\sigma,\ 
T,\ F,\ q,\ p\Big) \] 
for any linearly independent $X$ and $Y$ in $T\Sigma$.  
Now we may choose $\sigma$ and $\tilde{\sigma}$ such that 
$\sigma(x)\in s(x)$, $\tilde{\sigma}$ lies nowhere in a curvature 
sphere and  
\[ d\sigma= d\tilde{\sigma} \circ S,\]
where $S\in \Gamma \text{End}(T\Sigma)$. Now since $\tilde{\sigma}$ 
lies nowhere in a curvature sphere, we have that if 
$Z\in T_{x}\Sigma$ such that $d_{Z}\tilde{\sigma} \in \mathcal{F}(x)$, 
then $Z=0$.  On the other hand, since $\sigma(x)\in s(x)$, 
we have that $(d\sigma)_{x}\in T_x \Sigma \otimes 
\mathcal{F}(x)$. Therefore we must have that $S(x)=0$. Writing 
\[ S = \pmt{\alpha& \beta\\ \gamma &\delta}\]
in terms of the basis $X,Y$ of $T\Sigma$,
we may rewrite~(\ref{eq:areafcn}) as 
\[
\begin{array}{l}
\det \Big(\!
\big((\sigma,p)\!-\! \alpha(\tilde{\sigma},p)\big)d_{X}\tilde{\sigma}
 \!-\! \beta(\tilde{\sigma},p)d_{Y}\tilde{\sigma},\,
 \big((\sigma,p)\!-\!\delta(\tilde{\sigma},p)\big)d_{Y}\tilde{\sigma} 
 \!-\! \gamma (\tilde{\sigma},p)d_{X}\tilde{\sigma},\, 
 T, F, q, p\Big) \\[2mm]
= \Big(\big((\sigma,p)- \alpha(\tilde{\sigma},p)\big)
\big((\sigma,p)-\delta(\tilde{\sigma},p)\big) 
- \beta\gamma(\tilde{\sigma},p)^{2}\Big) 
\det(d_{X}\tilde{\sigma}, d_{Y}\tilde{\sigma}, T,F, q,p).
\end{array}\]
Since $\sigma$ is nowhere a curvature sphere, similar arguments as 
in the proof of Lemma~\ref{lem:signprop} show that 
$\{d_{X}\sigma,d_{Y}\sigma, T,F,q,p\}$ is a basis for 
$\Sigma\times \R^{4,2}$. Thus~(\ref{eq:areafcn}) is proportional to 
\[ \lambda = \big((\sigma,p)- \alpha(\tilde{\sigma},p)\big)
\big((\sigma,p)-\delta(\tilde{\sigma},p)\big) 
- \beta\gamma(\tilde{\sigma},p)^{2}.\]
At the umbilic point $x$, since $S(x)=0$, we have that 
$\lambda(x) = (\sigma(x),p)^{2}$. Hence, $\lambda(x)=0$ 
if and only if $s(x)\perp p$. So let us now assume that $s(x)\perp p$. 
Then  
\[
\begin{array}{rcl}
d_{X}\lambda &=& 
\big(\beta(d_{Y}\tilde{\sigma},p)- (d_{X}\alpha)(\tilde{\sigma},p)\big)
\big((\sigma,p)-\delta(\tilde{\sigma},p)\big) \\
&&\hspace{5mm}+ \big((\sigma,p)- \alpha(\tilde{\sigma},p)\big)
\big(\beta(d_{Y}\tilde{\sigma},p)
+ (\alpha-\delta)(d_{X}\tilde{\sigma},p) 
- (d_{X}\delta)(\tilde{\sigma},p)\big)\\
&&\hspace{10mm}-d_{X}(\beta\gamma)(\tilde{\sigma},p)^{2} 
-2\beta\gamma(d_{X}\tilde{\sigma},p)(\tilde{\sigma},p)\\
d_{Y}\lambda &=& (\gamma (d_{X}\tilde{\sigma},p) + 
(\delta - \alpha)(d_{Y}\tilde{\sigma},p) 
- (d_{Y}\alpha)(\tilde{\sigma},p)((\sigma,p)-\delta(\tilde{\sigma},p))\\
&&\hspace{5mm}+ ((\sigma,p)- \alpha(\tilde{\sigma},p))
( \gamma (d_{X}\tilde{\sigma},p) - (d_{Y}\delta)(\tilde{\sigma},p))\\
&&\hspace{10mm}-d_{Y}(\beta\gamma)(\tilde{\sigma},p)^{2} 
- 2\beta\gamma (d_{Y}\tilde{\sigma},p)(\tilde{\sigma},p).
\end{array}
\]
Using that $S(x)=0$ and $(\sigma,p)=0$ it is then clear that 
$(d\lambda)_{x}=0$. We shall now consider the Hessian of $\lambda$ at $x$. 
\[\begin{array}{rl}
(d_{X}d_{X}\lambda)_{x}&\!\!\!
= 2\big((d_{X_{x}}\alpha) (d_{X_{x}}\delta) 
- (d_{X_{x}}\beta)(d_{X_{x}}\gamma)\big)(\tilde{\sigma}(x),p)^{2}\\ 
(d_{X}d_{Y}\lambda)_{x} 
&\!\!\!
=
 \Big((d_{X_{x}}\alpha) (d_{Y_{x}}\delta) 
\!+\! (d_{Y_{x}}\alpha) (d_{X_{x}}\delta )
- (d_{Y_{x}}\beta) (d_{X_{x}}\gamma)
-(d_{Y_{x}}\gamma) (d_{X_{x}}\beta)\Big)(\tilde{\sigma}(x),p)^{2}\\
(d_{Y}d_{Y}\lambda)_{x}&\!\!\!= 2\big((d_{Y_{x}}\alpha)( d_{Y_{x}}\delta) 
- (d_{Y_{x}}\beta)(d_{Y_{x}}\gamma)\big)(\tilde{\sigma}(x),p)^{2}
\end{array}\]
On the other hand, by using these special lifts $\sigma$ 
and $\tilde{\sigma}$ one can compute the discriminant of 
$\mathcal{C}_{x}(\sigma,\tilde{\sigma})$ to be a positive scalar 
multiple of $(\det \Hess \lambda)_{x}$. Therefore we have 
arrived at the following theorem:

\begin{theorem}\label{thm:singcondu}
Let\/ $x$ be an umbilic point
of\/ $\mathcal{F}$ such that\/ $\lambda(x)=0$.
Then\/ $f$ has a\/ $D_4^+$ singularity\/
$($respectively, $D_4^-$ singularity\/$)$ at\/ $x$
if and only if the discriminant of\/ 
$\mathcal{C}_{x}$ is negative\/ $($respectively, positive\/$)$.
\end{theorem}

\section{Singularities of Lie sphere transformations}
\label{sec:singliesp}

Let $\mathcal{F}:\Sigma\to\mathcal{Z}$ be a
Legendre immersion, and let $A\in O(4,2)$.
Let $f:\Sigma\to \R^3$ be the 
projection of $\mathcal{F}$ to $\R^3$ in the manner of
\eqref{eq:projeceuc} and \eqref{eq:lifts}, and assume that $f$ is 
immersed.
We described Lie sphere transformations 
$\hat f :\Sigma\to \R^3$ 
of $f$ in Section 
\ref{sec:LiePrelims}, \eqref{eq:liesptrdef}
and \eqref{eq:liesptrdefeuc}, and we shall now study 
singularities appearing on such transformations.

\subsection{Non-umbilic points}
Let $x$ be a non-umbilic point
of $f$ and assume that $\hat\lambda(x)=0$, 
where $\hat{\lambda}$ denotes the signed area density function of $\hat{f}$.
Let $\kappa_1$ and $\kappa_2$ be the principal curvatures of $f$.
Then by Theorem~\ref{thm:singcond}, we have the following:
\begin{theorem}\label{thm:singcondl}
Let\/ $(u,v)$ be curvature line coordinates for\/
$f$ and assume that\/ $\hat{f}_{u}=0$ at\/ $x$
$($thus\/ $x$ is a rank\/ $1$ singular point\/$)$. Then 
\begin{enumerate}
\item $\kappa_{1,u}\neq 0$ at\/ $x$ if and only if\/ 
$\hat{f}$ has a type\/ $1$ singularity at\/ $x$,
\item $\kappa_{1,u}=0$ and\/ $\kappa_{1,uu}\neq 0$ at\/ $x$ 
if and only if\/ $\hat{f}$ has a type\/ $2$ singularity at\/ $x$,
\item $\kappa_{1,u}=\kappa_{1,uu}=0$ and\/ $\kappa_{1,uuu}\neq 0$ 
at\/ $x$ if and only if\/ $\hat{f}$ has a type\/ $3$ singularity at\/ $x$.
\end{enumerate}
\end{theorem}
\begin{proof}
The theorem follows by noting that $T +\kappa_{1}F$ is a 
lift of the curvature sphere $s_{1}$ of $\mathcal{F}$. Thus,
$AT+\kappa_{1}AF$ is a lift of the curvature sphere $As_{1}$ 
of $A\mathcal{F}$. The result then follows by applying 
Theorem~\ref{thm:singcond} to this lift.
 \end{proof}

We remark that in~\cite{FH} and \cite{porteous} 
similar conditions were derived for singularities of parallel surfaces, 
however in that case the 
singularities are precisely determined, allowed for
by the fact that 
parallel transformations are a strict subgroup of the full
group of Lie sphere transformations.  Note that when taking 
parallel transformations of a regular surface there are 
at most two parallel surfaces with 
singularities at a given point.  However, in the case of Lie sphere 
transformations there are infinitely many such surfaces.  

\subsubsection{Varying subclasses using Lie sphere transformations} 

In Theorem~\ref{thm:singcondl} we showed that the type of 
singularities appearing on a Lie sphere transformation is 
determined by the initial surface $f$. Using 
Theorem~\ref{thm:singcond} we will now show that we can explicitly 
construct Lie sphere transformations of the initial surface 
that give each possible subclass of the determined class. 
%Here we treat type $2$ degenerate singularities as 
%a single subclass and likewise for type $3$ degenerate singularities. 

To start with, we would obviously like our Lie sphere transformation 
$\hat{f}$ to have a singularity at $x$. To simplify matters we would 
also like the $u$ derivative of $\hat{f}$ to vanish at $x$. 
To achieve this we choose a unit timelike vector $\hat{p}\in \R^{4,2}$ 
such that $s_{1}(x)\perp \hat{p}$. There are many of these to choose 
from since $s_{1}(x)^{\perp}/s_{1}(x)$ has signature $(3,1)$. 
Now since $\hat{p}$ is a unit timelike vector there exists (many) 
$A\in O(4,2)$ such that $A\hat{p} = p$. If we then consider the 
Legendre immersion $A\mathcal{F}$, we have by the 
orthogonality of $A$ that 
$As_{1}(x)\perp A\hat{p}=p$, 
and thus $\hat{f}$ satisfies $\hat{f}_{u}=0$ at $x$. 

Immediately by Theorem~\ref{thm:singcondl}, we see that if our initial 
surface satisfies $\kappa_{1,u}\neq 0$ at $x$ then our Lie sphere 
transformation $\hat{f}$ has a cuspidal edge at $x$.  

For type $2$ and $3$ singularities, one can see from 
Theorem~\ref{thm:singcond} that we must impose some additional 
constraints on $\hat{p}$. If our initial surface satisfies 
$\kappa_{1,u}=0$ at $x$ then by considering the lift 
$\sigma_{1}=T+\kappa_{1}F$, it is clear that 
$(\sigma_{1,u})_{x}\in s_{1}(x)$. We then have that 
$W:=\span\{\sigma_{1}(x),d\sigma_{1}(T_{x}\Sigma)\}^{\perp}/s_{1}(x)$ 
has dimension $3$ and signature $(2,1)$. 
Thus we may choose a unit timelike vector $\hat{p}$, satisfying
\begin{equation}\label{eq:hatpw}
\hat{p}\, \operatorname{mod}\, s_{1}(x)\in W.
\end{equation}
This time by choosing $A\in O(4,2)$ such that $A\hat{p}=p$, 
we have that $A\mathcal{F}$ satisfies $As_{1}(x)\perp p$ and,
by \eqref{eq:hatpw},
$$
\inner{d(A\sigma_{1})_{x}}{p}
=
\inner{Ad(\sigma_{1})_{x}}{A\hat{p}}
=
\inner{d(\sigma_{1})_{x}}{\hat{p}}
=0,
$$
and so
\begin{equation}\label{eq:dsigmap}
d(A\sigma_{1})_{x}\perp p.
\end{equation}
Now if we are in the case that $\kappa_{1,u}=0$ and 
$\kappa_{1,uu}\neq 0$ at $x$, then by Theorem~\ref{thm:singcondl} 
we know that $\hat{f}$ has a type $2$ singularity at $x$.
Moreover, \eqref{eq:dsigmap} implies
that $\hat{f}$ has a degenerate singularity at $x$,
and by Theorem~\ref{thm:singcond}, we can see that $\hat{f}$ has 
either cuspidal lips, cuspidal beaks or some other less well known 
type $2$ degenerate 
singularity at $x$. In order to specify exactly which one of these
we have, let $\hat{p}^{\xi}:= \hat{p} + \xi \sigma_{1}$, 
for some $\xi \in\R$. Then $\hat{p}^{\xi}$ is again a unit timelike 
vector whose quotient lies in $W$ and, furthermore, one can check that 
\[ \Hess(\sigma_{1},\hat{p}^{\xi})_{x} 
= \Hess(\sigma_{1},\hat{p})_{x} + \xi c,\]
for some non-zero constant $c$. Therefore, by varying $\xi$, any 
sign of $\Hess(\sigma_{1},\hat{p}^{\xi})_{x}$ can be achieved. 
Now by letting $A^{\xi}\in O(4,2)$ such that $A^{\xi}\hat{p}^{\xi}=p$, 
we have by the orthogonality of $A^{\xi}$ that 
$\Hess(A^{\xi}\sigma_{1},p)_{x}= \Hess(\sigma_{1},p^{\xi})_{x}$. 
Then, by Theorem~\ref{thm:singcond}, by varying $\xi$ we change between 
cuspidal beaks, cuspidal lips and a less familiar type $2$ 
degenerate singularity at $x$. 

In the case that $\kappa_{1,u}=0$, $\kappa_{1,uu}=0$ and 
$\kappa_{1,uuu}\neq 0$ at $x$, we have by Theorem~\ref{thm:singcondl} 
that $\hat{f}$ projects to a type $3$ singularity at $x$, 
and by Theorem~\ref{thm:singcond} we have that $\hat{f}$ has a type 
$3$ degenerate singularity at $x$. 

On the other hand, in the case that $\kappa_{1,u}=0$, $\kappa_{1,uu}\neq 0$ 
(respectively, $\kappa_{1,u}=\kappa_{1,uu}=0$ and $\kappa_{1,uuu}\neq 0$) 
at $x$ if we wish to obtain a swallowtail (respectively, cuspidal butterfly) 
at $x$, the additional constraint on $\hat{p}$ we impose is that 
the quotient $\hat{p}\, \operatorname{mod}\, s_{1}(x)\not\in W$. 
Now by letting $A\in O(4,2)$ such that $A\hat{p}=p$, we have that 
$A\mathcal{F}$ satisfies $As_{1}(x)\perp p$ and 
$d(A\sigma_{1})_{x}\not\perp p$. Therefore, by Theorem~\ref{thm:singcond}, 
$\hat{f}$ has a swallowtail (respectively, cuspidal butterfly) at $x$.  

To illustrate the freedom given by Lie sphere transformations, 
we give the following example: 

\begin{example}\label{firstexample}
The surface in $\R^3$ given by 
\[ f(u,v)=(u,u^2,v) \] 
satisfies $k_{1,u}=0$ and $\kappa_{1,uu}\neq 0$ at $(u,v)=(0,0)$. 
Let $\hat{f}$ be a Lie sphere transformation of $f$ determined by 
$A\in O(4,2)$. Then by Theorem~\ref{thm:singcondl}, if $\hat{f}_{u}=0$ 
at $(0,0)$ then $\hat{f}$ must have a type $2$ singularity at $(0,0)$. 
For example, $\hat{f}$ is a swallowtail at $(0,0)$
if we use the Lie sphere transformation   
\[
A=\begin{pmatrix}
\tfrac{-1}{2} & 0 & 0 & 0 & \tfrac{-1}{2} & 1 \\
0 & 1 & 0 & 0 & 0 & 0 \\
0 & 0 & 1 & 0 & 0 & 0 \\
0 & 0 & 0 & 1 & 0 & 0 \\
\tfrac{1}{2} & 0 & 0 & 0 & \tfrac{-3}{2} & 1 \\
1 & 0 & 0 & 0 & -1 & 1 
\end{pmatrix}.
\]
Whereas, if we use, with $\chi = \sqrt{1+2 \xi^2}$, 
\[
A^{\xi}=\begin{pmatrix}
- \chi^{-1} (\tfrac{1}{\sqrt{2}}+ \xi) & 0 & 0 & 0 & 0 & \chi^{-1} 
(\tfrac{1}{\sqrt{2}} - \xi) \\
 \chi^{-1} \xi (-1 +\sqrt{2}\xi) & -\tfrac{\chi}{\sqrt{2}} & 0 & 
 -\tfrac{\chi}{\sqrt{2}} & 0 & -\chi^{-1} \xi (1+\sqrt{2} \xi) \\
0 & 0 & 1 & 0 & 0 & 0 \\
0 & \tfrac{1}{\sqrt{2}} & 0 & -\tfrac{1}{\sqrt{2}} & 0 & 0 \\
0 & 0 & 0 & 0 & 1 & 0 \\
\tfrac{1}{\sqrt{2}}-\xi & \xi & 0 & \xi & 0 & \tfrac{1}{\sqrt{2}} 
+ \xi
\end{pmatrix} \; , 
\]
we will have a cuspidal beaks (respectively, cuspidal lips) at $(0,0)$ if 
$\xi < 1/(2 \sqrt{2})$ ($\xi > 1/(2 \sqrt{2})$).  When 
$\xi=1/(2 \sqrt{2})$, we obtain a less familiar type $2$
degenerate singularity. Figure \ref{fig:Axi}
shows how the singularity changes type as $\xi$ varies.  
\end{example}

\begin{figure}[phbt]
\begin{center}
\begin{tabular}{ccc}
\includegraphics[width=0.3\linewidth,bb=0 0 295 432]{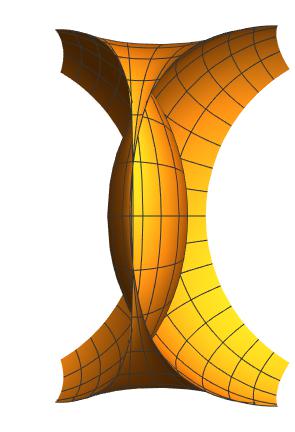} 
\quad\quad\quad&
\includegraphics[width=0.12\linewidth,bb=0 0 118 432]{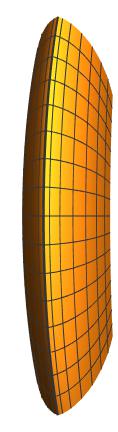} 
\quad\quad\quad&
\includegraphics[width=0.25\linewidth,bb=0 0 247 432]{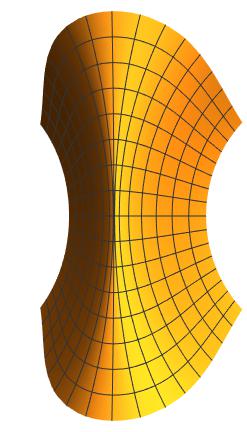} 
\end{tabular}
\caption
{From left to right, cuspidal beaks ($\xi < 1/(2 \sqrt{2})$), 
other type $2$ degenerate singularity ($\xi = 1/(2 \sqrt{2})$) and 
cuspidal lips ($\xi > 1/(2 \sqrt{2})$), as in Example \ref{firstexample}}
\label{fig:Axi}
\end{center}
\end{figure} 

%\begin{figure}[phbt]
%\label{firstfigure}
%\begin{center}
%\begin{tabular}{ccc}
%\includegraphics[width=0.4\linewidth]{figs/firstfig-swallowtail.eps} &
%\includegraphics[width=0.24\linewidth]{figs/firstfig-beaks.eps} &
%\end{tabular}
%\begin{tabular}{ccc}
%\includegraphics[width=0.3\linewidth]{figs/firstfig-lips.eps} &
%\includegraphics[width=0.3\linewidth]{figs/firstfig-degenerate.eps} &
%\end{tabular}
%\caption
%{a cuspidal beaks ($\xi < 1/(2 \sqrt{2})$), as in 
%Example \ref{firstexample}}
%\end{center}
%\end{figure}

\subsection{Umbilic case}
Now suppose that $x$ is an umbilic point of $\mathcal{F}$. 
Then for any Lie sphere transformation $A\in O(4,2)$, $x$ is 
an umbilic point of $A\mathcal{F}$. 
Let $\sigma, \tilde{\sigma}\in \Gamma \mathcal{F}$ such that 
$\sigma(x)\in s(x)$ and $\tilde{\sigma}(x)\not\in s(x)$. 
Then, if we let $\sigma^{A}:= A\sigma$ and 
$\tilde{\sigma}^{A}:= A\tilde{\sigma}$, 
we have that $\sigma^{A}(x) \in As(x)$ and 
$\tilde{\sigma}^{A}(x) \not\in As(x)$. 
We may then compute the cubic form 
$\mathcal{C}^{A}_{x}(\sigma^{A},\tilde{\sigma}^{A})$ 
for $A\mathcal{F}$ in terms of these lifts, but since 
$A$ is constant and orthogonal it is clear that 
$\mathcal{C}^{A}_{x}(\sigma^{A},\tilde{\sigma}^{A})
=\mathcal{C}_{x}(\sigma,\tilde{\sigma})$. 
By applying this to the spaceform projections 
$f:\Sigma\to \R^{3}$ and $\hat{f}:\Sigma\to \R^{3}$, 
we obtain the following theorem as a corollary of 
Proposition~\ref{prop:umbilic} and Theorem~\ref{thm:singcondu}: 

\begin{theorem}\label{lastthm}
Suppose that\/ $x$ is an elliptic\/ $($respectively, hyperbolic\/$)$ 
umbilic point of\/ $f$. 
If\/ $\hat{f}$ is immersed at\/ $x$ then\/ $x$ is an elliptic\/ 
$($respectively, hyperbolic\/$)$ umbilic point of\/ 
$\hat{f}$. 
Otherwise, $\hat{f}$ has a\/ $D^{-}_{4}$ $($respectively, 
$D^{+}_{4}$ singularity\/$)$ at\/ $x$. 
\end{theorem}
This result extends the result of \cite{FH} for parallel 
transformations to the group of Lie sphere transformations.

%\section{thebibliography}

\begin{flushright}
\begin{tabular}{ll}
\begin{tabular}{l}
(Pember)\\
Institut f\"ur Diskrete Mathematik \\
und Geometrie\\
Forschungsgruppe Differentialgeometrie \\
und Geometrische Strukturen\\
Technische Universit\"at Wien\\
Wiedner Hauptstra\ss e 8-10/104\\
A-1040 Wien, Austria\\
{\tt masonO\!\!\!ageometrie.tuwien.ac.at}\\
\end{tabular}
&
\begin{tabular}{l}
Department of Mathematics, \\
Kobe University,
Rokko 1-1, Nada, \\
Kobe 657-8501, Japan\\
(Rossman)\\
{\tt wayneO\!\!\!amath.kobe-u.ac.jp}\\
(Saji)\\
{\tt sajiO\!\!\!amath.kobe-u.ac.jp}\\
(Teramoto)\\
{\tt teramotoO\!\!\!amath.kobe-u.ac.jp}\\
\end{tabular}
\end{tabular}
\end{flushright}


\begin{thebibliography}{9}
{\small
\bibitem{AGV} V. I. Arnold, S. M. Gusein-Zade, and A. N. Varchenko,
{\it Singularities of differentiable maps}, Vol. $1$,
Monographs in Mathematics {\bf 82}, Birkh\"auser, Boston, 1985.

\bibitem{acous}
V. I.  Arnol'd,
{\it Singularities of caustics and wave fronts},
Math. and its Appl. {\bf 62} 
Kluwer Academic Publishers Group, Dordrecht, 1990.

\bibitem{blaschke}
W. Blaschke,
{\it Vorlesungen {\"u}ber Differentialgeometrie {III}},
Springer Grundlehren XXIX,
Berlin,
1929.

\bibitem{BHR1}
F. E. Burstall, U. Hertrich-Jeromin and W. Rossman,
{\it Lie geometry of flat fronts in hyperbolic space},
C. R. Math. Acad. Sci. Paris {\bf 348} (2010), no. 11-12, 661--664. 

\bibitem{BHR2}
F. E. Burstall, U. Hertrich-Jeromin and W. Rossman,
{\it Lie geometry of linear {W}eingarten surfaces}, 
C. R. Math. Acad. Sci. Paris
{\bf 350} (2012), no. 7-8, 413--416. 

\bibitem{cecil}
T. E. Cecil,
{\em Lie sphere geometry: With applications to submanifolds},
Second edition,
Universitext. Springer, New York, 2008.

\bibitem{clarke}
D. J. Clarke,
{\it Integrability in submanifold geometry},
PhD Thesis, University of Bath, 2012.

\bibitem{ferapontov1}
   E. V. Ferapontov,
     {\it Lie sphere geometry and integrable systems},
    Tohoku Math. J.
    {\bf 52} (2000), no. 2, 199--233.  

\bibitem{ferapontov2}
 E. V. Ferapontov,
     {\it The analogue of {W}ilczynski's projective frame in {L}ie sphere
              geometry: {L}ie-applicable surfaces and commuting
              {S}chr\"odinger operators with magnetic fields},
   Internat. J. Math. {\bf13} (2002), no. 9, 959--985. 

\bibitem{fsuy}
         S. Fujimori, K. Saji, M. Umehara and K. Yamada,
         {\em  Singularities of maximal surfaces},
         Math. Z. {\bf 259} (2008), 827--848. 

\bibitem{FH} 
T. Fukui, and M. Hasegawa, {\em Singularities of parallel surfaces}, 
\textit{Tohoku Math. J. {\bf64}} (2012), 387--408.

\bibitem{ishima}
G. Ishikawa and Y. Machida,
{\it Singularities of improper affine spheres and surfaces of 
constant Gaussian curvature},
Internat. J. Math. {\bf 17} (2006), no. 3, 269--293.

\bibitem{izdual}
S. Izumiya,
{\it Legendrian dualities and spacelike hypersurfaces in the lightcone},
Moscow Math. J. {\bf 9} (2009), no. 2, 325--357,

\bibitem{mandala}
S. Izumiya and K. Saji,
{\em The mandala of Legendrian dualities for pseudo-spheres 
in Lorentz-Minkowski space and "flat'' spacelike surfaces},
J. Singul. {\bf 2} (2010), 92--127.

\bibitem{horoflat}
S. Izumiya, K. Saji and M. Takahashi,
{\em Horospherical flat surfaces in hyperbolic\/ $3$-space},
J. Math. Soc. Japan {\bf 62} (2010), 789--849.

\bibitem{harm}
M. Kokubu,
{\it Surfaces and fronts with harmonic-mean curvature one 
in hyperbolic three-space},
Tokyo J. Math. {\bf 32} (2009), no. 1, 177--200.

\bibitem{krsuy}
  M. Kokubu, W. Rossman, K. Saji, M. Umehara and K. Yamada,
  \newblock{\em Singularities of flat fronts in hyperbolic\/ $3$-space},
  \newblock Pacific J. Math. {\bf 221} (2005), no. 2, 303--351.

\bibitem{orien}
M. Kokubu and M. Umehara,
{\it Orientability of linear Weingarten surfaces, 
spacelike CMC-1 surfaces and maximal surfaces},
Math. Nachr. {\bf 284} (2011), no. 14-15, 1903--1918.

\bibitem{lie}
S. Lie,
{\it {\"{U}}ber {K}omplexe, insbesondere {L}inien-und {K}ugelkomplexe, 
mit {A}nwendung auf der {T}heorie der partieller {D}ifferentialgleichungen},
Math. Ann. {\bf 5} (1872), 145--208, 209--256.

\bibitem{luy}
H. Liu, M. Umehara and K. Yamada,
{\it The duality of conformally flat manifolds},
Bull. Braz. Math. Soc. (N.S.) {\bf 42} (2011), no. 1, 131--152. 

\bibitem{mu}
S. Murata and M. Umehara,
{\it Flat surfaces with singularities in Euclidean\/ $3$-space},
J. Differential Geom. {\bf 82} (2009), no. 2, 279--316.

\bibitem{mussonic}
E. Musso and L. Nicolodi,
     {\it Deformation and applicability of surfaces in {L}ie sphere
              geometry}, Tohoku Math. J. {\bf 58} (2006), no. 2, 161--187. 

\bibitem{nao}
K. Naokawa,
{\it Singularities of the asymptotic completion of 
developable M\"obius strips},
Osaka J. Math. {\bf 50} (2013), no. 2, 425--437.

\bibitem{mason}
M. Pember, {\it Special surface classes}, PhD Thesis, 
University of Bath, 2015. 

\bibitem{porteous}
  I. R. Porteous, 
  {\em Geometric differentiation. 
  For the intelligence of curves and surfaces}.
  Second edition. Cambridge University Press, Cambridge, 2001.

\bibitem{d4}
K. Saji, 
{\em Criteria for\/ $D_4$ singularities of wave fronts},
Tohoku Math. J. {\bf 63} (2011), 137--147.

\bibitem{suy3}
         K. Saji, M. Umehara and K. Yamada,
         {\em $A_k$ singularities of wave fronts},
         Math. Proc. Cambridge Philos. Soc. {\bf 146} (2009), 731--746.

\bibitem{coh}
K. Saji, M. Umehara and K. Yamada,
{\it Coherent tangent bundles and Gauss-Bonnet formulas for wave fronts},
J. Geom. Anal. {\bf 22} (2012), no. 2, 383--409.

\bibitem{t1}
K. Teramoto, 
{\em Parallel and dual surfaces of cuspidal edges}, 
Diff. Geom. Appl. {\bf 44} (2016), 52--62.

\bibitem{t2}
K. Teramoto, 
{\em Principal curvatures and parallel surfaces of wave fronts}, 
preprint, arXiv:1612.00577.

\bibitem{max}
M. Umehara and K. Yamada,
{\it Maximal surfaces with singularities in Minkowski space},
Hokkaido Math. J. {\bf 35} (2006), no. 1, 13--40.


}
\end{thebibliography}
\end{document}